\documentclass[a4paper, 12pt]{amsart}
\usepackage{amsmath, amssymb}
\usepackage{amsfonts,amscd}
\newtheorem{thm}{Theorem}
\newtheorem{defn}[thm]{Definition}
\newtheorem{prop}[thm]{Proposition}

\newtheorem{remark}[thm]{Remark}
\newtheorem{cor}[thm]{Corollary}

\newtheorem{quest}[thm]{Question}
\def\Z{\mathbb Z}
\def\N{\mathbb N}

\def\R{\mathbb R}
\def\C{\mathbb C}

\def\O{\operatorname{O}}

\def\log{\operatorname{log}}

\begin{document}
\title[Distribution of gaps between eigenvalues \dots]{Distribution of gaps between eigenangles of Hecke operators}
\author[Sudhir Pujahari]{Sudhir Pujahari}
\address{Sudhir Pujahari, Harish-Chandra Research Institute (HBNI), Chatnag Road, Jhunsi, Allahabad - 211019, Uttar Pradesh, India}
\email{sudhirpujahari@hri.ac.in}
\subjclass[2000]{Primary 11R42, 11S40, Secondary \vskip0.1mm11R29}
\keywords{Equidistribution, convolution, Hecke operators, Sato- \vskip0.1mm\hskip0mm Tate conjecture, Satake parameters, $GL_2$ representations, Primitive Maass \vskip0.1mmforms}
\maketitle
\begin{abstract}
In 1931, Van der Corput showed that if for each positive integer $s$, the sequence $\{x_{n+s}-x_n\}$ is uniformly distributed (mod 1), then the sequence $x_n$ is uniformly distributed (mod 1).  The converse of above result is surprisingly not true.  The distribution of consecutive gaps of an equidistributed sequence has been studied widely in the literature.  
In this paper, we have studied the distribution of gaps between one or more equidistributed sequences.  Under certain conditions, we could study the distribution effectively.  As applications, we study the equidistribution of gaps between eigenangles of Hecke operators acting on space of cusp forms of weight $k$ and level $N$, primitive Maass forms.  We also have studied the distribution of gaps between corresponding angles of Satake parameters of $GL_2$ with prescribed local representations. 
\end{abstract}
\section{Introduction}\label{Intro}
The rich story of equidistribution started in the years 1909-1910 by the work of P. Bohl~\cite{Bohl}, H. Weyl~\cite{Weyl1} and
W. Sierpinski~\cite{Sie}
where they studied the distribution of the sequence $\{n \alpha\},$ (for an irrational $\alpha$) on the unit interval.  Here, $\{x\}$ denotes the fractional part of $x.$ 
Let us recall that a sequence of real numbers $\{x_n\}$ lying in the interval 
$[0,1] \subseteq \R$ is said to be uniformly distributed or equidistributed with respect to Lebesgue measure if 
for any interval $[\alpha,\beta] \subseteq [0,1]$, we have 
$$\lim_{V \to \infty } \frac{1}{V} \# \{n \leq V: \{x_n\} \in [\alpha, \beta]\}=\beta-\alpha.$$  This subject attracted great attention of mathematicians from all branches of mathematics after Hermann Weyl related the study of equidistribution to the study of exponential sums in his 1916 paper~\cite{Weyl}.
Our work is partly motivated by the following result of Van der Corput (see~\cite[page no. 176]{RM}): 
 If for each positive integer $s,$ the sequence $\{x_{n+s}-x_n\}$ is uniformly distributed (mod 1), then the sequence $\{x_n\}$ is uniformly distributed (mod 1). 
We consider the following question:
 \begin{quest}
 Is the converse of Van der Corput's result true? 
 \end{quest}
  In other words, 
if $\{x_n\}$ is uniformly distributed (mod 1), then is it true that for any positive integer $s,$ the sequence $\{x_{n+s}-x_n\}$ is uniformly distributed (mod 1)?  \\
The answer to the above question is surprisingly no.  In fact we see an example of a uniformly distributed (mod 1) sequence such that for any subsequence, $\{x_{n+1}-x_n\}$ will not be uniformly distributed (mod 1).  
For example, consider the well-studied sequence $\{[n \alpha]\},$ $\alpha $ is irrational.  Write the sequence as follows:
For a natural number $N,$ define,
$$A_{\alpha}(N) =\{\alpha n \ \text{(mod) \ 1}: 1 \leq n \leq N\}\subset \{\{n \alpha\} \ (mod \ 1) : n \in \N\},$$ 
and write them as increasing order as follows: 
\begin{equation}\label{E2}
A_{\alpha}(N)=\{0 < x_1 \leq x_2 \leq x_3 \leq \dots \leq x_N < 1\}
\end{equation}
where $x_{N+1}=1+x_1.$  \\
In 1957, Steinhaus conjectured the following fact:
$$\# \{x_{i+1}-x_i: 1\leq i \leq N\} \leq 3,$$
where $\#$ denotes the cardinality of a set.
There are several proofs of the above conjecture available in the literature, but the first proof was given by Vera S\'os~\cite{Sos1} and~\cite{Sos2} in 1958.  The above statement is popularly known as ``The three gap theorem". 
In 2002, V\^{a}j\^{a}itu and Zaharescu~\cite{VZ} investigated the following question: 
\begin{quest}
Let $A_{\alpha}(N)$ be as defined in (\ref{E2}).  Remove as many elements of $A_{\alpha}(N)$ as one likes.
Then, how large is the cardinality of the consecutive differences of the resulting set? 
\end{quest}
More explicitly, they proved the following: \\
For any subset $\Omega$ of $A_{\alpha}(N),$ there are no more than $(2+\sqrt 2)\sqrt{N}$ distinct consecutive differences, that is,
if 
$$B(\Omega)=\{x_{i+1}-x_i: 1 \leq i \leq N, \ x_i \in \Omega\},$$
then
$$\#B(\Omega(N)) \leq  (2+\sqrt 2)\sqrt{N}.$$
In 2015, using additive combinatorics A. Balog, A.Granville and J. Solymosi~\cite{BGS} improved the bound of above result~\cite{VZ} to $2 \sqrt{2N}+1$ for any finite subset of $\R / \Z.$ 
 In particular for our concerned sequence, they proved that 
$$\#B(N) \leq 2 \sqrt{2N}+1.$$
From the above result, we can conclude that for any subsequence say $\{y_n\}=\{y_1,y_2,y_3,...\}$ of $\{x_n\}=\{x_1,x_2,x_3,...\}$, the consecutive difference $\{y_{i+1}-y_i\}$ is not uniformly distributed (mod 1).    
In this paper we show that, if a sequence is equidistributed in $[-\frac{1}{2},\frac{1}{2}]$ with respect to a probability measure say $\mu=F(x)dx$ (for definition see Section \ref{Equi}), then the fractional parts of gaps of all elements of the sequence will be equidistributed in $[0,1]$ with respect to the measure $F(x) * F(x) \ dx$, where $*$ is the convolution. 
\noindent More explicitly,
if we have two sequences say $\{x_n\}_{n=1}^{\infty}$ and $\{y_{m}\}_{m=1}^{\infty}$ such that they are equidistributed with respect to probability measures $\mu_1=F_1(x)dx$ and $\mu_2=F_2(x)dx$ respectively in $[-\frac{1}{2},\frac{1}{2}],$ then  
the sequence of fractional parts of gaps between elements of $\{x_n\}$ and $\{y_m\}$, that is, $\{x_n-y_m\}_{n,m=1}^{\infty}$ (mod 1) is equidistributed in $[0,1]$ with respect to $F_1(x)*F_2(x) dx.$ 
We are also able to predict quantitatively the rate of convergence of the following:
$$\lim_{V \to \infty} \frac{1}{V^2} \#\{1 \leq m,n \leq V: \{x_n-y_m\} \ \text{mod} \ 1 \in [\alpha, \beta]\},$$
where $[\alpha,\beta]$ is any subset of $[0,1]$, whenever  $\{x_n\}_{n=1}^{\infty}$ and $\{y_{m}\}_{m=1}^{\infty}$ satisfy some conditions that have been described in Theorem \ref{Main}.  More generally, we have similar results for $r$ equidistributed sequences.  These results are stated in Section \ref{Results} as Theorem \ref{T1}, \ref{Measure} and \ref{Main}.  
We have discussed several applications of our results in Section \ref{App}.  \\ \\
Let $S(N,k)$ be the space of all  holomorphic cusp forms of weight $k$ with respect to $\varGamma_0(N)$.
For any positive integer 
$n$, let $T_n(N,k)$ be the $n^{th}$ Hecke operator acting on $S(N,k)$. 
Let $s(N,k)$ denote the dimension of the vector space $S(N,k)$.  For 
a positive integer $n \geq 1$, let
 $$\left \{a_{i}(n), 1 \leq i \leq s(N,k)\right\}$$
 denote the eigenvalues of $T_n$,
counted with multiplicity.
 For any positive 
integer $n$, let $T_n^{'}$ be the normalized Hecke operator acting on $S(N,k)$, defined as follows
$$ T_n^{'}:= \frac{T_n}{n^{\frac{k-1}{2}}} .$$
Consider  
$$ \left \{ \frac{a_{i}(n)}{n^{\frac{k-1}{2}}}, 1 \leq i \leq s(N,k)\right \},$$  the eigenvalues of $T_n^{'}$ counted with multiplicity.  Let $p$ be a prime number such that $p$ and $N$ are coprime. Then by the theorem of
Deligne (see~\cite{Deligne}) proving the Ramanujan-Petersson inequality, we know that
$$a_{i}(p) \in [-2p^{\frac{k-1}{2}},2p^{\frac{k-1}{2}}].$$
For each $i,$ choose $\theta_{i}(p)\in [0,\pi]$ such that
$$\frac{a_{i}(p)}{p^{\frac{k-1}{2}}}= 2\cos \theta_{i}(p).$$

Using results of Murty and Sinha~\cite{MS}, Murty and Srinivas~\cite{MSr} have recently proved the following results
$$\#\left \{(i^{(1)},i^{(2)}), 1 \leq i^{(1)},i^{(2)} \leq s(N,k):  \theta_{i}^{(1)}(p)  \pm \theta_{i}^{(2)}(p) = 0 \right\}$$
$$= \O \left( (s(N,k))^2 \left(\frac{\log \,p}{\log kN} \right)\right).$$
Note that taking $k$ and $N$ sufficiently large, the above result gives a little evidence towards the Maeda and Tsaknias conjectures. 
 As applications of our theorems, we can get the measure with respect to which the differences of eigenangles of Hecke operators are equidistributed and as a special case, the above result recovers Theorem 1 in~\cite{MSr}.  We could also get an error term (see Section \ref{App}, Theorems \ref{pair-correlation}, \ref{T15}).  We discuss similar results for Hilbert modular forms (see Theorems \ref{T13}, \ref{T14}), primitive Maass forms (see Theorems  \ref{T7}, \ref{T24}).
 In the case of primitive Maass forms, we have assumed the Ramanujan bound. 
 \section{Statement of results}\label{Results}
Let us start with a result that predicts the Weyl limits of gaps between equidistributed families.  Definitions are provided in Sections 3 and Section 4.
Here onwards, we denote $\{x\}$ as fractional part of $x$.
$$\{x\}:= x- \lfloor x \rfloor,$$
where $\lfloor x \rfloor$ is the largest integer not grater then $x$.
 \begin{thm}\label{T1}
Consider $\{X_{n}^{(1)}\}_{n=1}^{\infty},\{X_{n}^{(2)}\}_{n=1}^{\infty},\dots,\{X_{n}^{(r)}\}_{n=1}^{\infty}$ to be collection of $r$ sequences of multisets in $[0,1]$ such that $1 \leq j \leq r$, $\# X_{n}^{(j)}\to \infty \ \ \text{as $n \to \infty.$}$  For every $m \in \Z,$ let $c_{m}^{(j)}$ be the $m^{th}$  Weyl limit of $X_{n}^{(j)}$ respectively, that is ,
$$c_{m}^{(j)}:= \lim_{{n}\to \infty}\frac{1}{\#X_{n}^{(j)}}\sum_{t\in X_{n}^{(j)} }e( m t).$$ 
Assume that $c_{m}^{(j)}$ exists for each $j.$
If $C_m$ is the $m^{th}$Weyl limit of the family
 $$ \{ \{x_1 + x_2 + \dots +x_r\} , \  x_j \in X_{n}^{(j)}, \  1 \leq j \leq r \}$$ that is, for $m \in \Z,$
$$C_m:=  \lim_{n  \to \infty} \nonumber
 \frac{1}{\prod_{j=1}^{r} \# X_{n}^{(j)} } \sum_{{x_j \in X_{n}^{(j)} }\atop{1 \leq j \leq r}}   
 e( m\{x_1 +  x_2 + \dots +   x_r\}). $$
 Then the Weyl limit 
\begin{equation}\label{E11}
C_m= \prod_{j=1}^r c_{m}^{(j)}.
\end{equation}
\end{thm}

\begin{remark}
To consider the gaps, we can take the family  
 $$ \{ \{x_1 - x_2 - \cdots -x_r\}, \  x_j \in X_{n}^{(j)}, \  1 \leq j \leq r \}.$$ in the above theorem.
\end{remark}
Further, if we consider the multisets such that if $x \in A_{n}^{(j)}$ then $-x \in A_{n}^{(j)}.$  For simplicity, let us write $A_{n}^{(j)}=\{\pm x_{j}\},$ we get the following corollary.
\begin{cor}
In particular, if we consider the family $A_{n}^{(j)}=\{\pm x_{j}\} \subseteq [-\frac{1}{2},\frac{1}{2}],$ 
then we have
 $$c_{m}^{(j)} = \lim_{{n}\to \infty}\frac{1}{\#A_{n}^{(j)}}\sum_{ t\in A_{n}^{(j)} }e(  m [ t]).$$ 
Let $C_m$ be the $m^{th}$ Weyl limit of the family
 $$ \{\{ x_1 + x_2 + \cdots + x_r\} , \  x_j \in A_{n}^{(j)}, \  1 \leq j \leq r \},$$ that is for $m \in \Z,$
$$C_m:=  \lim_{n  \to \infty} \nonumber
 \frac{1}{\prod_{j=1}^{r} \# A_{n}^{(j)} } \sum_{{x_j \in A_{n}^{(j)} }\atop{1 \leq j \leq r}}   
 e(m{\{ x_1 +  x_2 + \cdots +   x_r\}}). $$
 Then the Weyl limit 
\begin{equation}\label{E111}
C_m= \prod_{j=1}^r c_{m}^{(j)}
\end{equation}
\end{cor}
In the next theorem, we predict the measure with respect to which the above mentioned family in Theorem \ref{T1} is equidistributed. 
\begin{thm}\label{Measure}
 Consider $\{A_{n}^{(1)}\}_{n=1}^{\infty},\{A_{n}^{(2)}\}_{n=1}^{\infty},...,\{A_{n}^{(r)}\}_{n=1}^{\infty} \subseteq [-\frac{1}{2},\frac{1}{2}]$ to be sequences of multisets such that $-x_j \in A_{n}^{(j)}$ whenever $x_j \in A_{n}^{(j)}$,  and $\# A_{n}^{(j)}  \to \infty \ \ \text{as $n \to \infty$}$  for $j=1,2,3,...,r.$  
 If $\{A_{n}^{(j)} \}$, are equidistributed in $[-\frac{1}{2},  \frac{1}{2}]$ with respect to the measure $F_j(x)dx$ respectively, where
$$F_j(x)=\sum_{m=-\infty}^{\infty}c_{m}^{(j)}e(m x) ,$$ 
   then the family
$$\left \{{\{ x_1 + x_2 + \dots + x_r \}}, \  x_j \in A_{n}^{(j)}\right \}$$ is equidistributed in $[0,1]$ with respect to the measure 
$$\mu=F(x)dx,$$
where
$$F(x)= \sum_{m=-\infty}^{\infty}C_{ m }e(mx).$$ 
Moreover, if 
\begin{equation}\label{E0}
\sum_{m=-\infty}^{\infty}|c_{m}^{(j)}|^2< \infty  \ \ \text{for all $1 \leq j \leq r$},
\end{equation}
then the above function $F(x)$ equals
$$F_1 * F_2* \dots* F_r(x),$$
where
$$F_1 * F_2* \dots* F_r(y)$$
$$=\int_{0}^{1} \dots \int_{0}^{1}F_1(y_1)F_2(y_2) \dots F_r(y-y_1-y_2-\dots -y_{r-1})dy_{r-1}dy_{r-2}\dots dy_1.  $$
\end{thm}
\begin{remark}
In general an equidistributed family may not satisfy (\ref{E0}), but
in the families of interest to us, the Weyl limits satisfy (\ref{E0}).
\end{remark}
Let $\{A_{n}^{(j)}, 1 \leq j \leq r\}$ be equidistributed sequences of finite multisets.  If we know the distribution effectively, then
our next result helps us to predict the effective equidistribution of family of gaps of equidistributed families.
\begin{thm}\label{Main}
Let 
$\{A_{n}^{(1)}\}_{n=1}^{\infty},\{A_{n}^{(2)}\}_{n=1}^{\infty},...,\{A_{n}^{(r)}\}_{n=1}^{\infty}$ 
be sequences of finite multisets as defined in Theorem \ref{Measure}.
Let $\{A_{n}^{(j)}\}_{n=1}^{\infty}, 1 \leq j \leq r$ be equidistributed sequences in $[-\frac{1}{2},\frac{1}{2}]$ with respect to the measure $F_j(x)dx,$ where
$$F_j(x)= \sum_{m=-\infty}^{\infty}c_{m}^{(j)}e(mx)$$
and $c_{m}^{(j)}$ are as defined in Theorem \ref{T1}.  
Consider $\underline{x}=(x_1,x_2,..,x_r)$, $A_{\underline{n}}=A_{n}^{(1)} \times A_{n}^{(2)} \times \dots \times A_{n}^{(r)}$. 
Then, for any positive integer $M$ 
and any $I=[\alpha,\beta] \subseteq [0,1],$ we have
$$\left|\frac{1}{\prod_{j=1}^r (\#A_{n}^{(j)}) } \#  \left \{ \underline{x} \in
 A_{\underline {n}} : \left\{{ x_1 + x_2 + \dots + x_r}\right\} \in I \right \} 
- \int_{I} \mu  \right |$$
$$\leq \frac{\prod_{j=1}^{r} \# A_{n}^{(j)}}{M+1}
+  \sum_{|m| \leq M} 
 \left(\frac{1}{M+1}+min\left(\beta-\alpha, \frac{1}{\pi |m|} \right) \right) $$
 $$\left(\left|\prod_{j=1}^{r} \sum_{{x_j} \in A_{n}^{(j)}} e(m x_j )- \prod_{j=1}^{r} \# A_{n}^{(j)}c_{m}^{(j)} \right| \right), $$
where $\mu$ is as defined in Theorem \ref{Measure}.
\end{thm}
\begin{remark}
In the above theorems, $[-\frac{1}{2},\frac{1}{2}]$ can be taken to be any interval of length $1$.
\end{remark}
In the next few sections, we review basic facts that we use in the proofs of Theorem \ref{T1}, Theorem \ref{Measure} and Theorem \ref{Main}.  

\section{Equidistribution}\label{Equi}
\begin{defn}
A  sequence of real numbers $\{x_n\}$ lying in the interval 
$[0,1]$ is said to be uniformly distributed or equidistributed with respect to Lebesgue measure if 
for any interval $[\alpha,\beta] \subseteq [0,1]$, we have 
$$\lim_{V \to \infty } \frac{1}{V} \# \{n \leq V: x_n \in [\alpha, \beta]\}=\beta-\alpha.$$  
\end{defn}
In 1916, Weyl~\cite{Weyl1} proved that the sequence $\{n^2 a\}$, where $a$ is irrational is equidistributed in the unit interval.
In the same paper he gave a revolutionary criterion for uniform distribution in terms of exponential sums that now known as Weyl criterion.  Explicitly, it says that, a sequence $\{x_n\}$ is uniformly distributed in the unit interval if and only if
 for every $m\in \Z,\,m \neq 0,$
$$c_m:=\lim_{N\to \infty}\frac{1}{N}\sum_{n=1}^N e(mx_n) = 0,$$
where $e(t) = e^{2\pi i t} $ and the $c_m$ defined above is called the $m^{th}$ Weyl limit.
For proof of the above theorem see~\cite[page no. 172]{RM}.

\section{Set equidistribution}\label{Setequidistribution}
Consider finite multisets $A_n \subseteq [0,1]$ with $\#A_n \to \infty.$  The sequence $\{A_n\}$ is  set-equidistributed with respect to a probability measure $\mu$ in $[0,1]$ if for every $[\alpha, \beta]\subseteq[0,1],$
$$\lim_{n\to\infty}\frac{\#\{t\in A_n:\,t \in [\alpha,\beta]\}}{\#A_n} = \int_{\alpha}^{\beta} d\mu.$$

In this case, for every $m\in\Z,$ define { ``Weyl limits'':}
$$c_m:= \lim_{n\to \infty}\frac{1}{\#A_n}\sum_{t\in A_n}e(mt).$$
The following is a generalisation of the classical Wiener-Schoenberg criterion can be found in~\cite[page no. 195]{RM}.
\begin{thm}\label{SW}[Wiener-Schoenberg] \\
A sequence $\{A_n\}_{n=1}^{\infty} \subseteq [0,1]$ is equidistributed with respect to some positive continuous measure $F(x)dx,$ where $F(x) = \sum c_me(mx)$
if and only if  
$$c_m = \lim_{n\to \infty}\frac{1}{\#A_n}\sum_{t\in A_n}e(mt),$$
 exists for every integer $m$ and 
$$\lim_{V \to \infty} \frac{1}{V}\sum_{m=1}^V|c_m|^2 =0.$$ 
\end{thm}
\section{Fourier Analysis}

According to our need, let us recall some facts from Fourier Analysis in this section.  The reader may refer~\cite{Rudin} for detail study.
For our convenience, let us define:
$$e(x):=e^{2 \pi i x}.$$ 
Let $f$ be  a periodic and integrable function of period $1$ on $\R.$  The Fourier series of $f$ is given by 
$$f(x)= \sum_{n=- \infty}^{\infty}\hat{f}(n) e\left(n {x}\right),$$
where $\hat{f}(n)$ are called the Fourier coefficients, defined as  \\
$$\hat{f}(n):= \int_{0}^{1}f(x)e\left(-n {x} \right), \ n \in \Z.$$

Let $f_i, 1 \leq j \leq r$ be $r$  integrable function on $\R$ of period $1$.
Define the convolution 
of $r$ periodic integrable functions of period $1$ on $\R$, denoted as 
$$f_1*f_2* \dots *f_r: \R \rightarrow \C$$ as follows:
\begin{equation}\label{convolution}
f_1*\dots *f_r (y) 
\end{equation} 
$$=   \int_{0}^{1} \dots \int_{0}^{1}f_1(y_1)f_2(y_2) \dots f_r(y-y_1-y_2- \dots -y_{r-1})dy_{r-1}dy_{r-2} \dots dy_1.$$  
Among the many interesting properties that convolution of periodic integrable functions satisfies, the following property serves our purpose:
\begin{equation}\label{E20}
{(f_1 * f_2 * \dots * f_r)}^{\wedge}(n)
= \hat{f_1}(n) \hat{f_2}(n) \dots \hat{f_r}(n).
\end{equation}
In particular, for $r=2,$ 
$${(f_1*f_2)}^{\wedge}(n)=\hat{f_1}(n) \hat f_2(n).$$
The following theorem can be concluded from the famous Riesz-Fischer theorem (see~\cite[page no. 91]{Rudin}):
\begin{thm}\label{Riesz-Fischer}[Riesz-Fischer]\\
Let $\{a_n\}$ be a sequence of real numbers.
If $$\sum_{n=-\infty}^{\infty}|a_n|^2 < \infty,$$
 then there exists a unique periodic square Lebesgue integrable function $f$ that is $f \in L^2[0,1]$ such that 
$$f(x)= \sum_{n= -\infty}^{\infty} a_ne(nx),\ \ \text{where $\hat{f}(n)=a_n.$}$$
\end{thm}  
\section{Beurling-Selberg polynomials}\label{BS}
Let $\chi_I(x)$ be the characteristic function of the interval $[a,b] \subseteq \R$.
 For a positive integer $M$, define ${\vartriangle_M}(x)$ to be the Fejer's Kernel, defined as below: \\
$${\vartriangle_{M}}(x)= \sum_{|n|<M}\left(1-\frac{|n|}{M}\right)e(nx)=\frac{1}{M} \left( \frac{sin \ \pi Mx}{sin \ \pi x}\right)^2.$$
The $M^{th}$-order Beurling polynomial is defined as follows: \\

$${B_M}^{*}(x)=\frac {1}{M+1} \sum_{n=1}^M\left(\frac{n}{M+1}-\frac{1}{2}\right)\vartriangle_M\left(x-\frac{n}{M+1}\right)$$
$$+\frac{1}{2\pi(M+1)} \sin(2\pi(M+1)x)   
-\frac{1}{2\pi}\vartriangle_{M+1}(x) \sin \ 2 \pi x + \frac{1}{2(M+1)}\vartriangle_{M+1}(x).$$

For an interval $[a,b]$, the $M^{th}$ order Selberg polynomial is defined as below: \\

$$S^{+}_M(x)=b-a+B^{*}_M(x-b)+B^{*}_M(a-x).$$

$$S^{-}_M(x)=b-a+B^{*}_M(b-x)+B^{*}_M(x-a).$$

It is clear that both the above polynomials are trigonometric polynomials of degree at most $M$. 
From the work of Vaaler (see~\cite{Vaaler}), for all $M \geq 1,$ we have \\

\noindent \textbf{(a)} 
For a subinterval $[a,b]$ of $\R$, \\

$S_{M}^{-}(x) \leq \chi_I(x) \leq S_{M}^{+}(x) $. \\

\noindent \textbf{(b)} 
$S_{M}^{\pm}(x)=  \sum_{0\leq |m|\leq M} \hat S_{M}^{\pm}(m)e(mx)$, 
where \\

$\hat S_{M}^{\pm}(0)=b-a \pm \frac{1}{M+1}.$ \\ 

For $0<|m| <M$, \\

$|\hat {S}_M^{\pm}(m)| \leq \frac{1}{M+1}+ min \left \{b-a, \frac{1}{\pi|m|}\right \}.$ \\

\noindent \textbf{(c)} 
$||S_M^{+}-\chi(x)||_{L_1} \leq \frac{1}{M+1}.$  \\

\noindent \textbf{(d)}
For $n \neq 0$, note that
$$|\hat{\chi_{I}}(n)|=\left|\frac{\sin \ \pi n (b-a)}{\pi k} \right| \leq \text{min} \left(b-a, \frac{1}{\pi |n|} \right).$$ 

For $0< |n| < M$,
$$|\hat{S}^{\pm}_M(n)| \leq \frac{1}{M+1}+ \text{min} \left(b-a,\frac{1}{\pi |n|} \right).$$

\section{Proofs of Theorems}\label{proof}

\noindent\textbf{Proof of Theorem \ref{T1}}.  \\  
By the definition of the Weyl limits, we know

$$ C_m :  =  \lim_{n \to \infty} 
 \frac{1}{\prod_{j=1}^{r} \# X_{n}^{(j)} } \sum_{{x_{j} \in X_{n}^{(j)} } \atop {1 \leq j \leq r}} 
   e\left(m\left\{{ x_{1} + x_{2} + \dots + x_{r}}\right\}\right)  $$ 
 Observe that 
 $$e(mx)=e(m\{x\}).$$
 Using the above observation, we have
\begin{eqnarray*}
 C_m & = & \lim_{n  \to \infty}\frac{1}{\prod_{j=1}^{r} \# X_{n}^{(j)} } \prod_{j=1}^r \sum_{x_j \in X_{n}^{(j)}}e( m x_j) \\ 
 & = & \prod_{j=1}^{r}c_{m}^{(j)}. 
  \end{eqnarray*}
 In particular, for $A_{n}^{(j)}= \{\pm x_j\} \subseteq [-\frac{1}{2}, \frac{1}{2}]$ the above calculation follows immediately, that is
 for any non zero integer $m$, if $c_{m}^{(j)}$ be the $m^{th}$ Weyl limit of the family $\{\pm x_j, \  x_j \in X_{n}^{(j)}, \ 1 \leq j \leq r\}$ and $C_m$ be the $m^{th}$ Weyl limit of the family $\{\{ x_1 + x_2 + \dots + x_r \}, \ x_j \in A_{n}^{(j)}, \  1 \leq j \leq r\}$, then
 $$C_m=\prod_{j=1}^{r}c_{m}^{(j)}.$$
\noindent \textbf{Proof of Theorem \ref{Measure}}. \\  
Since $\{A_{n}^{(j)}\}$ are equidistributed in $[-\frac{1}{2}, \frac{1}{2}],$ with respect to the measure $F_j(x)dx,$ by Theorem \ref{SW} we have,
$$\lim_{V \to \infty} \frac{1}{V}\sum_{|m|\leq V} |c_{m}^{(j)}|^2=0  \  \text{for $1 \leq j \leq r$}.$$
 Hence, $|c_{m}^{(j)}|<1$, except possibly for finitely many $m$. \\
Now using Theorem \ref{T1} and above fact, we have 
$$|C_m| < |c_{m}^{(j)}|$$
 except possibly for finitely many $m$ and for all $1 \leq j \leq r. $  
 Hence,
\begin{equation}\label{E25}
\lim_{V \to \infty} \frac{1}{V}\sum_{|m|\leq V}|C_m|^2 =0.
\end{equation}
Hence, by Theorem \ref{SW}, we can conclude that   
$$\left\{\{ x_1 + x_2 +  \dots + x_r\}\right\}$$
 is equidistributed in $[0,1]$ with respect to the measure 
$$\mu=F(x)dx,$$
where 
$$F(x)dx=\sum_{m=-\infty}^{\infty} C_me(mx).$$
In addition, if the concerned family satisfies (\ref{E0}), that is
  $$\sum_{m=-\infty}^{\infty} |c_{m}^{(j)}|^2 < \infty, 1 \leq j \leq r,$$
  then $|c_{m}^{(j)}|<1,$ except possibly for finitely many $m .$ \\ 
  Using Theorem \ref{T1} again, we have    
 $$\sum_{m=-\infty}^{\infty}|C_m|^2 < \infty.$$
  Hence, by Theorem \ref{Riesz-Fischer},
 there exists a function $F \in L^{2}([0,1])$ such that 
 $$F(x)=\sum_{m}C_m e(mx).$$
 But note that 
 $$C_m=\prod_{j=1}^rc_{m}^{(j)} \ \text{and} \ \prod_{j=1}^{r}\hat {F_j}(m)
 = {(F_1*F_2* \dots *F_r)}^{\wedge}(m).$$
 Hence,
 $$ F(x)=  F_1 * F_2* \dots * F_r(x).$$
 
\noindent \textbf{Proof of Theorem \ref{Main}}. \\  
In~\cite{Montgomery} Montgomery gives a proof of the classical Erd\"{o}s-Tur\'{a}n inequality using Beurling-Selberg polynomials (see~\cite[Theorem 11.4.8]{RM}).  Generalizing the idea implicit in Montgomery's work, Murty and Sinha have proved a varient of Erd\"{o}s-Tur\'{a}n inequality (see~\cite[Theorem 8]{MS}).  Following the same path we prove Theorem \ref{Main}.
Let $\chi_{I}$ be the characteristic function of the interval $I$.
Then by \textbf{(a)} of Section \ref{BS}, we have
$$\sum_{{x_j \in A_{n}^{(j)} \atop {1 \leq j \leq r}}} S_M^{-}(x_{\underline n }) \leq
 \sum_{{x_j \in A_{n}^{(j)} \atop {1 \leq j \leq r}}}\chi_{I}
(x_{\underline n })
\leq \sum_{{x_j \in A_{n}^{(j)} \atop {1 \leq j \leq r}}} S_{M}^{+}(x_{\underline n }).$$
Now using the Fourier expansion of $S_M^{\pm}(x_{\underline n})$, we know that
$$\sum_{{x_j \in A_{n}^{(j)}}\atop {1 \leq j \leq r}} S_{M}^{\pm}(x_{\underline n })
=  \sum_{|m| \leq M} \hat {S}_{M}^{\pm}(m)
\left( \prod_{j=1}^r \sum_{x_j \in A_{n}^{(j)}} e(\pm m x_j ) \right).$$ 
Subtracting $\prod_{j=1}^r \#A_{n}^{(j)} c_{m}^{(j)} $ from the inner exponential sums, we get
\begin{equation}\label{E42}
  \sum_{x_j \in A_{n}^{(j)}\atop {1 \leq j \leq r}} S_{M}^{\pm} ( x_{\underline n})
- \left(\prod_{j=1}^r \# A_{n}^{(j)}\right) \sum_{|m| \leq M} {\hat{S}}_{M}^{\pm}(m) C_m  
\end{equation}
$$ =    \sum_{|m| \leq M} \hat {S}_{M}^{\pm}(m) \left( \prod_{j=1}^r  \sum_{x_j \in A_{n}^{(j)}} e(\pm m x_j )- \prod_{j=1}^r \#A_{n}^{(j)} c_{m}^{(j)} \right).$$

Since $c_{i_0}=1,$ 
$$ \left( \prod_{j=1}^r  \sum_{x_j \in A_{n}^{(j)}} e(\pm m x_j )- \prod_{j=1}^r \#A_{n}^{(j)} c_{m}^{(j)} \right)=0 \ \text{for $m=0$}.$$
Taking the absolute value on both sides we get

$$  \left | \sum_{x_j \in A_{n}^{(j)}\atop {1 \leq j \leq r}} S_{M}^{\pm} ( x_{\underline n})
-\prod_{j=1}^r \# A_{n}^{(j)} \sum_{|m| \leq M} {\hat{S}}_{M}^{\pm}(m) C_m  \right |$$
$$\leq   \sum_{|m| \leq M} |\hat {S}_{M}^{\pm}(m) | \left | \prod_{j=1}^r \sum_{x_j \in A_{n}^{(j)}}e( \pm m x_{i})- \prod_{j=1}^r \#A_{n}^{(j)} c_{m}^{(j)} \right | .$$

Now let us consider the sum 
$$\sum_{|m| \leq M} \hat {S}_{M}^{\pm}(m) C_m.$$
Since for all $|m| >M$, $\hat {S}_{M}^{\pm}(m)=0,$
 without loss of generality let us extend the range of sums to $\Z.$ Then, we have
 $$\sum_{m}\hat S_{M}^{\pm}(m)C_m=\sum_{m} C_m \int_{0}^1 
 S_{M}^{\pm}(x)e(-m x)dx.$$
 Now interchanging the sum and integral and using the definition of $\mu$, the above quantity equals 
$$  \int_{0}^1 S_{M}^{\pm}(x) d \mu.$$
 Using \textbf{(c)} of Section \ref{BS}, we have
\begin{equation}\label{E41}
 \left|\int_{0}^1 (S_{M}^{\pm}(x)-\chi_{I}(x)) d \mu \right| \leq \frac{\|\mu\|}{M+1}.
 \end{equation}
 Note that
 $$D_{I,V_n}(\mu)= \left|N_{I}(V)-\left(\prod_{j=1}^r A_{n}^{(j)}\right) \mu (I) \right|$$
 
 $$=\left |\sum_{x_j \in A_{i_r}\atop {1 \leq j \leq r}} \chi_I(x_{\underline n})- \prod_{j=1}^r (\# A_{n}^{(j)}) \int_0^1 \chi_{I}(x)d \mu \right|.$$
 Now adding and subtracting  $\prod_{j=1}^{r}(\# A_{n}^{(j)}) \int_{0}^1S_M^{+}(x) d \mu$ to the above expression, we get
  $$= \left|\prod_{j=1}^r \sum_{x_j \in A_{n}^{(j)}} \chi_{I}(x_{\underline n})-\prod_{j=1}^{r}(\# A_{n}^{(j)}) \int_{0}^1S_M^{+}(x) d \mu \right.$$
  $$\left. +\prod_{j=1}^{r}(\# A_{n}^{(j)}) \int_{0}^1S_M^{+}(x) d \mu-\prod_{j=1}^r(\# A_{n}^{(j)})\int_{0}^1\chi_I(x)d \mu \right |.$$
 Using triangle inequality, we get 
  $$D_{I,V}(\mu) \leq \left| \prod_{j=1}^r \sum_{x_j \in A_{n}^{(j)}} \chi_I(x_{\underline n})-\prod_{j=1}^{r}(\# A_{n}^{(j)}) \int_{0}^1S_M^{+}(x) d \mu \right|$$
  $$+\left| \prod_{j=1}^{r}(\# A_{n}^{(j)}) \int_{0}^1S_M^{+}(x) d \mu-\prod_{j=1}^r( \# A_{n}^{(j)})\int_{0}^1\chi_I(x)d \mu \right|$$
  $$\leq \left| \prod_{j=1}^r \sum_{x_j \in A_{n}^{(j)}} \chi_I(x_{\underline n})-\prod_{j=1}^{r}(\# A_{n}^{(j)}) \int_{0}^1S_M^{+}(x) d \mu \right|$$
  $$+\left|\prod_{j=1}^r (\# A_{n}^{(j)}) \int_{0}^1 (S_M^{+}(x)- \chi_{I}(x))d \mu \right|.$$
  Now using (\ref{E41}), the above is 
  $$\leq \frac{\prod_{j=1}^r(\# A_{n}^{(j)} \|\mu\|)}{M+1}
  +\left| \prod_{j=1}^r\sum_{x_j \in A_{n}^{(j)}} \chi_{I}(x_{\underline n})- \prod_{j=1}^r(\# A_{n}^{(j)})\int_{0}^1 S_M^{+}(x)d \mu\right|.$$
  Using \textbf{(a)} of Section \ref{BS}, we have 
  $$\leq \frac{\prod_{j=1}^r(\# A_{n}^{(j)} \|\mu\|)}{M+1}$$
  $$+ \left|\prod_{j=1}^r \sum_{x_j \in A_{n}^{(j)}}S_M^{+}(x_{\underline n})-\prod_{j=1}^r \# A_{n}^{(j)} \int_{0}^1 S_M^{+}(x)d \mu \right|.$$ 
  Using (\ref{E42}) and the fact that 
  $$\int_{0}^1 S_M^{+}(x) d \mu = \sum_{|m| \leq M } \hat{S}_M^{+}(m) C_m,$$
  $D_{I,V}(\mu)$ is
  $$ \leq \frac{\prod_{j=1}^r(\# A_{n}^{(j)} \|\mu\|)}{M+1}+ \sum_{|m| \leq M} \hat {S}_{M}^{\pm}(m) \left( \prod_{j=1}^r  \sum_{x_j \in A_{n}^{(j)}} e(\pm m x_j )- \prod_{j=1}^r \#A_{n}^{(j)} c_{m}^{(j)} \right) .$$
  Now using \textbf{(b)} of Section \ref{BS}, we have 
  $$\left|D_{I,V}(\mu)\right| \leq \frac{\prod_{j=1}^r(\# A_{n}^{(j)} \|\mu\|)}{M+1}$$
  $$+ \sum_{|m| \leq M} \frac{1}{M+1}+ min\left(b-a, \frac{1}{\pi |m|} \right)
  \left | \prod_{j=1}^r \sum_{x_j \in A_{n}^{(j)}}e(m x_j)- \prod_{j=1}^r (\# A_{n}^{(j)}) C_m\right |.$$
  Since $\|\mu\|=1$, we get the required result.
\section{Applications}\label{App}
In this section, we give several applications of our Theorems.  
For the first application $(i)$, we use the notations from Section \ref{Intro}.  \\ \\ 
\noindent\textbf{(i)}  
\noindent\textbf{Distribution of gaps of eigenangles of normalized Hecke operators $T_p^{'}$ acting on the space of cusp forms of level $N$ and weight $k$.}  \\  \\
The recently proved Sato-Tate conjecture in a series of papers by Barnet-Lamb, Geraghty, Harris, Shepherd-Barron 
 and Taylor~\cite{BL},~\cite{LMR},~\cite{MNR} says that
if $1 \leq i \leq s(N,k)$ be the eigenvalue of a non CM cusp form, then the family 
$\{a_{i}(p)\}$ is equidistributed in $[-2,2]$ as  $p \to \infty$ and $(p,N)=1$ with respect to the 
Sato-Tate measure
$$d\mu_{\infty}=F(x)dx,$$
where $$F(x)=\frac{1}{2 \pi} \sqrt {4-x^2}.$$
In 1997, Serre~\cite{Serre} studied the ``vertical" Sato-Tate conjecture by fixing a prime $p$ and varying $N$ and $k.$
In particular, 
he proved the following result: 
Let $N_{\lambda},k_{\lambda}$ be positive integers such that $k_{\lambda}$ is even, 
$N_{\lambda}+k_{\lambda} \to \infty$ and $p$ is 
a prime not dividing $N_{\lambda}$ for any ${\lambda}$.  Then the family of eigenvalues, 
$$\{a_{i}(p), 1 \leq i \leq s(N,k)\}$$
of the normalized $p^{th}$ Hecke
operator 
$$T_{p}^{'}(N_{\lambda},k_{\lambda})= \frac{T_{p}(N_{\lambda},k_{\lambda})}{p^{\frac{k_{\lambda}-1}{2}}}$$
is equidistributed in the interval $\Omega=[-2,2]$ with respect to the measure
$$\mu_p:=F(x)dx,$$
where $$F(x)=\frac{p+1}{\pi}\frac{\sqrt{1-\frac{x^2}{4}}}{(p^{\frac{1}{2}}+p^{-\frac{1}{2}})^2-x^2}.$$
\begin{remark}
Also in 1997, Conrey, Duke and Farmer~\cite{CDF}  studied a special case of above result by fixing $N=1.$ 
\end{remark}
In 2009, Murty and Sinha~\cite{MS} obtained the effective version of above result.  Precisely, they proved the following:  
Let p be a fixed prime.  Let $\{(N,k)\}$ be a sequence of pairs of positive integers such that $k$ is even, $p$ does not divide $N$ and $N+k \to \infty$.  For any interval $[\alpha,\beta] \subseteq [-2,2]$ and a pair $N,k$,
$$\frac{1}{s(N,k)} \# \left \{1 \leq i \leq s(N,k):a_{i}(p) \in [\alpha,\beta] \right \}
=\int_{\alpha}^{\beta}\mu_{p} +\O \left(\frac{\log \, p}{\log \, kN} \right).$$
Let $n=s(N,k)$ and for all $1 \leq j \leq r,$
$$A_{n}^{(j)}=\left \{\frac{\theta_{i}^{(j)}(p)}{2 \pi}, \ 1\leq i \leq s(N,k) \right \}.$$
Note that $$\# A_{n}^{(j)}=n=s(N,k).$$
So as $N+k \to \infty,$ we have $n \to \infty$ that is we have infinite number of multi sets.  Since each sets are same and we are going to study the distributions of gaps between the elements of the multisets $A_{n}^{(j)}.$
$$A_{n}^{(j)}=\left \{\frac{\theta_{i}^{(j)}(p)}{2 \pi}, \ 1\leq i \leq s(N,k) \right \},$$
where $\theta_{i}^{(j)}(p) \in [0,\pi]$ such that $a_{i}^{(j)}(p)= 2 \cos \theta_{i}(p).$
As applications of our theorems, we have following theorem: 
\begin{thm}\label{pair-correlation}
Let $N$ be a positive integer and $p$ a prime number not dividing $N$.  For any interval 
$[\alpha,\beta] \subseteq [0,1]$, $r \leq s(N,k),$  
$$\frac{1}{s(N,k)^r}\#\left\{1 \leq i^{(1)},...,i^{(r)} \leq s(N,k):\, 
{ \left\{\frac{\pm \theta_{i}^{(1)}(p)  \pm \dots\pm \theta_{i}^{(r)}(p)}{2 \pi}\right\}} \in [\alpha,\beta] \right\}$$
$$= \int_{[\alpha,\beta] }\nu_p + \O \left(\frac{\log \, p}{ \log \, kN} \right),$$
where
$$\nu_p = F(x) * F(x) *\dots* F(x)dx,$$
and
 $$F(x)=4(p+1) \frac{\sin ^{2} 2 \pi x}{\left(p^{\frac{1}{2}}+p^{-\frac{1}{2}}\right)^2- \cos ^2 2 \pi x}.$$
 Here the implied constant is effectively computable.
\end{thm}
\begin{remark}
For $r=2,$ the above mentioned measure 
$$\nu_p=\frac{2(1+\cos4 \pi x)(1-\frac{1}{p^2})+\frac{4}{p}( \frac{1}{p^2}-\cos4\pi x)}
{1+\frac{1}{p^4}-\frac{2}{p^2}\cos4\pi x} dx.$$
\end{remark}
 The following theorem can be deduced from Theorem \ref{pair-correlation}.
 \begin{thm}\label{T15}
 For any $\alpha \in [0,1],$
$$\#\left\{1 \leq i^{(1)},i^{(2)},\dots,i^{(r)} \leq s(N,k):\, 
{ \left\{\frac{\pm \theta_{i}^{(1)}(p)  \pm \dots\pm \theta_{i}^{(r)}(p)}{2 \pi}\right\}} = \alpha \right \}$$
$$= \O \left( (s(N,k))^r\left(\frac{\log \, p}{\log \,kN} \right)\right),$$
where the implied constant is effectively computable.
\end{thm}
In the above theorem for $r=2$, we have an interesting consequence, namely  
\begin{thm}\label{T12}
 For any $\alpha \in [0,1],$
$$\#\left \{(i^{(1)},i^{(2)}):  \left\{\frac{\pm \theta_{i}^{(1)}(p) \pm \theta_{i}^{(2)}(p)}{2 \pi} \right\} = \alpha \right \}= \O \left(  (s(N,k))^2\left(\frac{\log \,p}{\log kN} \right)\right).$$
\end{thm}

Using the fact that
$$\# \left \{(i^{(1)},i^{(2)}):  \left(\frac{\pm \theta_{i}^{(1)}(p) \pm \theta_{i}^{(2)}(p)}{2 \pi} \right) = \alpha \right \}$$
$$ \leq \# \left\{(i^{(1)},i^{(2)}):  \left\{\frac{\pm \theta_{i}^{(1)}(p) \pm \theta_{i}^{(2)}(p)}{2 \pi} \right\} = \alpha \right \},$$
we have the following corollary, 
\begin{cor}
For any $\alpha \in [0,1],$
$$\#\left \{(i^{(1)},i^{(2)}):  \left(\frac{\pm \theta_{i}^{(1)}(p) \pm \theta_{i}^{(2)}(p)}{2 \pi} \right) = \alpha \right \}= \O \left(  (s(N,k))^2\left(\frac{\log \,p}{\log kN} \right)\right).$$
\end{cor}
\begin{remark}
which recovers Theorem 1 in~\cite{MSr}, when $\alpha=0$.
\end{remark}

 \begin{remark}
In the above theorem if $\alpha=0$ and $N=1$ it gives evidence towards the famous Maeda conjecture and for $N \geq 1$, it gives evidence towards Tsaknias conjecture as mentioned in the introduction.
\end{remark}

Before proving Theorem \ref{pair-correlation} let us collect the following facts: \\
From~\cite{Serre} and~\cite{MS} we know that 
$$\left \{\frac{\pm \theta_{i}^{(j)}(p)}{2 \pi}, 1 \leq i \leq s(N,k) \right \}$$ is equidistributed in $[-\frac{1}{2}, \frac{1}{2}]$ with respect to the measure 
$$\mu_p=F(x)dx,$$
where 
\begin{equation}\label{E50}
F(x)=4(p+1) \frac{\sin ^{2} 2 \pi x}{\left(p^{\frac{1}{2}}+p^{-\frac{1}{2}}\right)^2- 4 \cos ^2 2 \pi x}.
\end{equation}
So using Theorem \ref{E11}, we can conclude that, the concerned family is equidistributed in $[0,1]$ with respect to the measure $\underbrace{F* F\dots* F}_\text{$r$ times} dx$.  \\
To proceed further, let us prove the following proposition: 
\begin{prop}\label{P10}
For any positive integer $m,$
$$\left |\prod_{j=1}^r \sum_{{i}\atop{1 \leq i \leq s(N,k) }}2 \cos m \theta_{i}^{(j)}(p)- C_m(s(N,k))^r \right|$$
$$\ll p^{\frac{3rm}{2}}m^r(\log \,p)^r 2^{r \nu(N)}+(\sqrt{N}d(N))^r,$$
where $\nu(N)$ is the number of distinct prime factor of $N$ and $d(N)$ is the divisor function.
The constant is absolutely and effectively computable.
\end{prop}

\begin{proof}
Estimating each terms of Eichler-Selberg trace formula Murty and Sinha (see~\cite{MS} Theorem 18 and (11))  proved the following: 
For any positive integer $m$, let $c_{m}^{(j)}$ be the Weyl limits of the family $$\{\pm \theta_{i}^{(j)}(p), 1 \leq i \leq s(N,k)\}.$$
For $1 \leq j \leq r$, the Weyl limits $c_{m}^{(j)}$ are given by 
 \begin{equation}\label{E12}
c_{m}^{(j)} =  
\left \{ 
\begin{array}{l l}
  1 &  \text{if } m=0 \\
  \left(\frac{1}{p^{\frac{m}{2}}}-\frac{1}{p^{\frac{m-2}{2}}}\right) &  \text{if $m$ is even }  \\
  0 &  \text{otherwise.}
\end{array}      
\right.
\end{equation}
    Moreover, for $m \geq 1$ and $1 \leq j \leq r,$
    $$\left|\sum_{i=1}^{s(N,k)}  (2 \cos m \theta_{i}^{(j)}(p) -c_{m}^{(j)} (s(N,k))\right| 
    \ll p^{\frac{3m}{2}}2^{\nu(N)} \log \,p^m+ d(N)\sqrt{N}.$$
Using the fact
\begin{equation}\label{EChev}
2 \cos m \theta_{i}^{(j)}(p)= X_m(2 \cos \theta_{i}^{(j)}(p))- X_{m-2}(2 \cos \theta_{i}^{(j)}(p)),  \ \ m \geq 2,
\end{equation}
where
$$X_m(2 \cos \theta)=\frac{\sin (m+1) \theta}{\sin \theta}$$
we have 
$$\prod_{j=1}^r \sum_{{i}\atop{1 \leq i \leq s(N,k) }}2 \cos m \theta_{i}^{(j)}(p)=(Tr T^{'}_{p^m}-Tr T^{'}_{p^{m-2}})^r.$$ 

Now, using the estimates of Eichler-Selberg trace formula that have been done in~\cite[page no. 696]{MS} and the well known inequality \\
\begin{equation}\label{E17}
(a-b)^r \leq r(a^r+b^r),
\end{equation} 
we have 
$$(Tr T^{'}_{p^m}-Tr T^{'}_{p^{m-2}})^r \ll r \left(\left(\frac{k-1}{12}\right) \psi(N)\left(\frac{1}{p^{\frac{m}{2}}}-\frac{1}{p^{\frac{m-1}{2}}}\right)\right)^r $$
$$+ \left(p^{\frac{3m}{2}}m(\log \,p) 2^{ \nu(N)}+\sqrt{N}d(N)\right)^r.$$
Using (\ref{E17}) and the fact that $s(N,k)=\O(kN),$ we get
$$\left |\prod_{j=1}^r \sum_{{i}\atop{1 \leq i \leq s(N,k) }}2 \cos m \theta_{i}^{(j)}(p)- (c_m(s(N,k)))^r \right|$$
$$\ll p^{\frac{3rm}{2}}m^r(\log \,p)^r 2^{r \nu(N)}+(\sqrt{N}d(N))^r.$$
\end{proof}
\noindent \textbf{Proof of Theorem \ref{pair-correlation}}.  \\ 
Using Theorem \ref{Main}, the concerned quantity is
$$\ll \frac{(s(N,k))^r}{M+1}
+p^{\frac{3rM}{2}}M^r(\log \,p)^r 2^{r \nu(N)}+(\sqrt{N}d(N))^r.$$
 Now we want to choose $M$ such that 
 $$\frac{(s(N,k))^r}{M+1} \sim p^{\frac{3rM}{2}}.$$
 And that can be achieved by Choosing 
$M=c \frac{\log \,kN}{\log \,p}$ for a sufficiently small constant $c$.
Putting the above value of $M$, we have the required result.  \\  \\
\noindent \textbf{(ii)}
\textbf{Distribution of gaps between the Satake parameters of $GL_2$ with prescribes local representations.} \\ \\
In this section, we follow notations and presentation from~\cite{LLW}.
Here we apply our theorems to study the distribution of gaps between eigenvalues of normalised Hecke operator $T^{'}_{p}$ acting on Hilbert modular forms and further to some $GL_2$ automorphic representations, whose local components  at a finite set of finite places are specified. 
Let $F$ be a totally real number field with degree $d \geq 2$.  
Let $\sigma_1, \sigma_2,...,\sigma_r$ be the  embeddings of $F$ into $\R$ with  valuations $\infty_1, \infty_2,..., \infty_r$ respectively. 
Let $\mathcal{O}$ be the ring of integers of $F.$ For any non-archimedean valuation $v,$ let $\wp_{v}$ denote the corresponding prime ideal and $q_v=N(\wp_v)$. 
Let $\mathbb A=\mathbb A_{F}$ be the set of adeles and $\mathbb A_{fin}$ be the set of finite adeles.  For any $\alpha=(\alpha_v)_{v< \infty} \in \mathbb A_{fin},$ the norm $N(\alpha)$ is defined as 
$$N(\alpha):=\prod_{v < \infty}N(\wp_v)^{ord_v \ \alpha_v}.$$

For any integral ideal $\mathfrak{a}$ define the ideal norm of  $\mathfrak{a}$ as 
$$N_{\mathfrak{a}}:=|\mathcal{O} / \mathfrak{a}|.$$
For $\alpha \in \mathcal{O},$ define the absolute norm of $\alpha $ as 
$$N(\alpha):= N ((\alpha)).$$

  Let $G=GL_2$ and ${Z}$ be its center.  For our convenience, let us write $\bar{G}={Z} \setminus {G}$ and $\bar{G}({\mathbb{A}})= {Z}(\mathbb{A}) \setminus G(\mathbb{A}).$  Let $L^2 \left( \bar{G}(F) \setminus \bar{G}(\mathbb{A})\right)$ be the space of square integrable function on $\bar{G}(F) \setminus  \bar{G}(\mathbb{A}).$
Let $L_0^2$ be the subspace of cuspidal functions.  Let $\textbf{n}$ and $\mathfrak{N}$ be two integral ideal of $\mathcal{O}.$  Define
$$M(\textbf{n}_v, \mathfrak{N}_v):= 
\left \{ \gamma= \begin{pmatrix} a & b \\
c & d 
\end{pmatrix} \in M_2(\mathcal{O}_v): c \in \mathcal{O}_v, (det \ \gamma) \mathcal{O}_v=\textbf{n}_v \right \}.$$ 
 For any ring $R,$ let 
$$M_2(R):=\left\{  \begin{pmatrix} a & b \\
c & d 
\end{pmatrix}: a,b,c,d \in R \right\}.$$
Define 
$$K_o (\mathfrak{N}_v):= \left\{ \begin{pmatrix} a & b \\
c & d \end{pmatrix} \in K_v : c \in \mathfrak{N}_v  \right \}, \ \  K_o (\mathfrak{N}):=\prod_{v < \infty} K_o (\mathfrak{N}_v). $$
where 
$K_v=GL_2(\mathcal{O}_v).$
Let $\underline {k}=(k_1,k_2,...,k_r)$ be an $r$ tuples of even integers with $k_i \geq 4, \  i=1,2,...,r.$  
For any integral ideal $\mathfrak{N}$ of $\mathcal{O}$, consider $\Pi_{\underline {k}}(\mathfrak{N})$ be the set of cuspidal automorphic representations $\pi$
in $L_{0}^{2},$ such that  
$$\pi_{fin}=\hat{\otimes}_{v < \infty} \pi_v \ \ \text{contains a non zero  $K_0(\mathfrak{N})$ fixed vector.}$$
and
$$\pi_{\infty_{i}}= \pi _{k_i},  \text{for $i=1,2,...,r$},$$
where  $\pi_{k_i}$ be the discrete representation of $GL_2(\R)$ of weight $k_i$ (even positive integer) with trivial central character.
From~\cite{BJ}, we know that, the set $\Pi_{\underline {k}}(\mathfrak{N})$ is finite. 
Let $B$ be the set of upper triangular matrices of $GL_2$,
$$ \chi \left( 
\begin{pmatrix}
a & b  \\
  & d
\end{pmatrix}
  \right)=\left|\frac{a}{d}\right|_{v}^{\frac{1}{2}} \chi_{1}(a) \chi_{2}(d),$$
  where  $\chi_{1}, \chi_{2}$ are unramified characters of $F_v$.  
 For any finite unramified place $v$ of $\pi,$
$$\pi_{v}=Ind_{B{(F_v)}}^{G(F_{v})} \chi,$$
where $Ind_{B{(F_v)}}^{G(F_{v})}$ is the induce representation of ${B{(F_v)}}$ on ${G(F_{v})}.$
Define
$$\lambda_{v}(\pi)=\alpha_{1v}+\alpha_{2v},$$
where $\alpha_{i,v}=\chi_{i}(\varpi),$ $\varpi$ is any uniformizer of $F_v.$
In the literature the above mentioned $\alpha_{iv}, i=1,2$ are called Satake parameters. Using Ramanujan conjecture (see~\cite{D} and~\cite{L}), we know that 
$$\lambda_{v}(\pi) \in [-2,2].$$ 
Let $S=\{w_1,w_2,...,w_l\}$ be a set of non-archemidean valuations with the corresponding prime ideal say ${\mathfrak{q}}_i$ attached to $w_i$.
Let $\rho_{w_i}$ be a super cuspidal representation of ${Z(F_{w_i})} \setminus {GL_2(F_{w_i})},$
where $Z$ is the center of $GL_2(F_{w_i}).$ 
Let ${{\mathfrak{q}}_i}^{c_i}$ be the conductor of $\rho_{w_i}$ and $\mathfrak{M}=\prod_{i=1}^{l}\mathfrak{q}_i^{c_i}$.
Define 
$$\Pi_{\underline{k}}(\mathfrak{N}, \underline {\rho}):= 
\{\pi \in \Pi_{\underline{k}}(\mathfrak{N}):\pi_{w_i} \cong \rho_{w_i} \text{for $i=1,2,...,l$}\},$$
where $\underline{\rho}=(\rho_{w_1}, \rho_{w_2},...,\rho_{w_l}).$
Consider $\# \Pi_{\underline{k}}(\mathfrak{N}, \underline {\rho})=n$. \\
For $1 \leq i \leq n$ and $\pi_i \in \Pi_{\underline{k}}(\mathfrak{N}),$ write 
$$\lambda_{v}(\pi_i)=2 \cos \theta_{v}(\pi_i), \ \ \theta_{v}(\pi_i) \in [0,\pi].$$
From the work of Lau-Li-Wang~\cite{LLW} and Li~\cite{Li} we know that the family $\{\lambda_{v}(\pi_i)\}$
is equidistributed in the interval $[-2,2]$ with respect to the measure
$$d \mu_v(x)=\frac{N(\wp)+1}{\left(N(\wp)^{\frac{1}{2}}+N(\wp)^{-\frac{1}{2}}\right)^2-x^2} \ d\mu_{\infty}x,$$
where 
\begin{eqnarray*}
d \mu_{\infty}(x)=
\left\{
\begin{array}{l l}
\frac{\sqrt{1-\frac{x^2}{4}}}{\pi} & \text{for $x \in [-2,2]$}  \\
0 & \text{otherwise.}
\end{array}
\right.
\end{eqnarray*}
\begin{thm}\label{T13}
Let 
$$\hat{\pi}=(\pi_1,\pi_2,...,\pi_r)$$
 and
$${\widehat{\Pi}}_{\underline{k}}(\mathfrak{N},\underline {\rho})=\Pi_{\underline{k}} (\mathfrak{N},\underline{\rho}) \times \Pi_{\underline{k}}(\mathfrak{N},\underline {\rho}) \times \dots \times \Pi_{\underline{k}}(\mathfrak{N},\underline {\rho}).$$
There exists a small constant $\delta>0,$ such that for all sufficiently large positive integers $n,$
$$\frac{1}{(\# \Pi_{\underline{k}}(\mathfrak{N} \mathfrak{M}), \underline{\rho})^r}
 \# \left\{ \hat{\pi}_v \in {\widehat{\Pi}}_{\underline{k}}(\mathfrak{N},\underline {\rho}):  \left\{\frac{\pm \theta_{v}(\pi_1) \pm \theta_{v}(\pi_2) \pm \dots\pm \theta_{v}(\pi_r)}{2 \pi}\right\} \in [\alpha,\beta]\right\}$$
$$=\int_{\alpha}^{\beta} \nu_{v} + \O\left(\frac{\log  \  N \wp}{\log   \ C_{\underline {k}}N(\mathfrak{N})}\right)$$
holds uniformly for integer $r \geq 1$ and a prime ideal   $\wp$ with valuation $v$, which is not in $S$, satisfying 
$$ r \log N(\wp) \leq \delta \log (C_{\underline k} N(\mathfrak{N})),$$  
and uniformly for any $[\alpha, \beta] \subseteq [0,1],$ and $1 \leq r \leq n.$
where, 
$$C_{\underline {k}}=\prod_{j=1}^r\frac{k_i-1}{4 \pi},$$
$$\nu_v=\underbrace{F_v(x)*F_v(x)* \dots *F_v}_\text{$r$ times}(x)dx, $$ 
and 
$$F_v=4(N(\wp)+1)\frac{sin^2 2 \pi x}{\left(N(\wp)^\frac{1}{2}+N(\wp)^{-\frac{1}{2}}\right)^2-4 x^2}.$$
Here the implied constant is effectively computable.
\end{thm}
Taking $S =\phi,$ we have the following corollary:
\begin{cor}\label{C3}
Let $\wp$ be a prime ideal in $\mathcal{O}.$  Suppose $\mathfrak{N}$ is an integral ideal with $(\mathfrak{N}, \wp)=(1).$  
Then,
for any $[\alpha, \beta] \subseteq [0,1],$
$$\frac{1}{(\# \Pi_{\underline{k}}(\mathfrak{N} ))^r} \# \left\{\hat{\pi}_v \in {\widehat{\Pi}}_{\underline{k}}(\mathfrak{N},\underline {\rho}): \left\{\frac{\pm \theta_{v}(\pi_1) \pm \theta_{v}(\pi_2) \pm \dots\pm \theta_{v}(\pi_r)}{2 \pi} \right\} \in [\alpha,\beta]\right\}$$
$$=\int_{\alpha}^{\beta} \nu_{\wp} + \O\left(\frac{\log \,  N \wp}{\log \, C_{\underline {k}}N(\mathfrak{N})}\right).$$
\end{cor}
Before proving Theorem \ref{T13} let us collect the following facts:
The following theorem can be seen as a special case of~\cite[Theorem 1.1]{LLW}: \\

Let $\wp$ be a prime ideal which is not in $S$ with valuation $v.$  For any $[\alpha, \beta] \subseteq [-\frac{1}{2},\frac{1}{2}],$ and $1 \leq r \leq n,$
$$\frac{1}{(\# \Pi_{\underline{k}}(\mathfrak{N} \mathfrak{M}))^r} \# \left\{1 \leq i \leq n: { \frac{\pm \theta_{v}(\pi_i)}{2 \pi}}\in [\alpha,\beta]\right\}$$
$$=\int_{\alpha}^{\beta} \nu_{\wp} + \O\left(\frac{\log \, N \wp}{\log \, C_{\underline {k}}N(\mathfrak{N})}\right),$$
where, $C_{\underline {k}}=\prod_{j=1}^r\frac{k_i-1}{4 \pi}.$  
In particular, $\left \{\frac{\pm \theta_{v}(\pi_i)}{2 \pi}\right \}$ is equidistributed in $[-\frac{1}{2}, \frac{1}{2}]$ with respect to the $\nu_{\wp}.$
To proceed further, let us prove the following proposition:
\begin{prop}\label{P11}
$$\left|\prod_{j=1}^r\sum_{\pi_i \in \Pi_{\underline{k}}(\mathfrak{N} \mathfrak{M})}  \cos 2 \pi m \theta_{v}(\pi_i)-(\# \varpi_{\underline{k}(\mathfrak{N} \mathfrak{M})})^r C_m \right|$$
$$\ll(d_{\underline \rho}N(\mathfrak{N})^{\frac{1}{2}+\epsilon}N(n)^{\frac{3}{2}})^r(N(m))^{2r},$$
\end{prop}
\begin{proof}
The following result can be deduced from~\cite[Proposition 8.1]{LLW}: \\
  For $1 \leq j \leq r,$
\begin{equation}\label{E15}
|\sum_{\pi_i \in \Pi_{\underline{k}}(\mathfrak{N} \mathfrak{M})}  \cos 2 \pi m \theta_{v}(\pi_i)-\# \varpi_{\underline{k}(\mathfrak{N} \mathfrak{M}))} c_{ i_m}|
\ll d_{\underline \rho}N(\mathfrak{N})^{\frac{1}{2}+\epsilon}N(n)^{\frac{3}{2}}N(m)^2,
\end{equation}
where
\begin{equation*}
c_{m}^{(j)}=
\left\{
\begin{array}{l l}
\frac{1}{2}\left(\frac{1}{N(\wp)^{\frac{|m|}{2}}}-\frac{1}{N(\wp)^{\frac{|m|-2}{2}}}\right) & \text{if $m$ is even}, \\
0 & \text{otherwise}.
\end{array}
\right.
\end{equation*}
Lau, Li and Wang proved the above result using Arthur's trace formula on $GL_2(F)$, where $F$ is a totally real algebraic number field of degree $\geq 2.$ 
Now using~\cite[Teorem 6.3]{LLW} and proceeding very similar way as proof of Proposition \ref{P10}, we can prove the above Proposition. 
\end{proof}

Now proceeding exactly like proof of Theorem \ref{pair-correlation} and choosing 
$$M=c \frac{ \log \, (C_{\underline k}N(\mathfrak{N}))}{\log \, \wp},$$
where $c$ is a sufficiently small constant
we can prove Theorem \ref{T13}.
The following theorem can be deduced from the above theorem.
\begin{thm}\label{T14}
For any $\alpha \in [0,1]$,
$$\# \left\{\bar{\pi} \in {\bar{\varpi}}_{\underline{k}}(\mathfrak{N}): \left\{\frac{\pm \theta_{v}(\pi_1) \pm \theta_{v}(\pi_2) \pm \dots\pm \theta_{v}(\pi_r)}{2 \pi}\right\}=\alpha \right \}$$
$$= \O\left( (\# \Pi_{\underline{k}}(\mathfrak{N} \mathfrak{M}))^r   \frac{\log \, N \wp}{\log (C_{\underline {k}}N(\mathfrak{N}))}\right),$$
where the implied constant is effectively computable.
\end{thm} 
\begin{remark}
In particular, when $r=2$ and $\alpha=0,$ we have 
similar result like Theorem \ref{T15} for Hilbert modular forms.  \\   
\end{remark}
 \noindent\textbf{(iii)}
\textbf{Distribution of gaps between eigenangles of Hecke operators acting on space of primitive Maass forms. }  \\  \\
In this section we follow the notations and presentations of~\cite{LY}.
Let $\mathcal{H}$ be the upper half plane in $\C.$  Consider $\Gamma=SL_2(\Z)$.  The non Euclidean Laplace
operator is given by
$$\bigtriangleup=-y^2\left(\frac{\delta^2}{\delta x^2}+\frac{\delta^2}{\delta y^2}\right).$$
The operator $\bigtriangleup$ is invariant under the action of $SL_2(Z)$ on $\mathcal{H},$ 
 where the action of $SL_2(\Z)$ on $\mathcal{H}$  defined as follows:  \\
 For any $z \in \mathcal{H}$ 
and $\gamma= 
\begin{pmatrix}
a & b \\
c & d
\end{pmatrix} \in SL_2(\Z),$  
$$\gamma z=\frac{az +b}{cz+d}.$$
\begin{defn}
A smooth function $f\neq 0$ on $\mathcal{H}$ is called a Maass form for the group $\Gamma$ if  \\
$(i)$ For all $ \gamma \in \Gamma $ and all $z \in \mathcal{H}, $
$$f(\gamma z)=f(z);$$
$(ii)$ f is an eigen function of above $\bigtriangleup$ that is, 
$$\bigtriangleup f = \lambda f,$$ 
$(iii)$  There exists a positive integer $N$ such that
$$f(z)
\ll y^{N} \ \ \text{as} \  y \to \infty$$
\end{defn}
We know that the Maass cusp forms span a subspace $C(\Gamma   \setminus \mathcal{H})$ in $L^2(\Gamma \setminus \mathcal{H}),$ where $L^2(\Gamma \setminus \mathcal{H})$ denote the square integrable function on $C(\Gamma \setminus \mathcal{H}).$
For any positive integer $n$, the Hecke operator $T_n$ together with Laplacian $\bigtriangleup$ forms a commutative family $H$ of Hermitian operators on $L^2(\Gamma \setminus \mathcal{H})$ with respect to the inner product
$$<f,g>=\int_{\Gamma /  \mathcal{H}}\frac{f(z) \bar{g}(z)}{y^{2}}dx dy.$$
 
Consider $\{u_i: i \geq 0\}$ to be a complete orthonormal basis for the subspace $C(\Gamma \setminus \mathcal{H})$ consisting of the simultaneous eigenfunctions on $\mathcal{H},$ where $u_0$ is a constant function.
Then 
$$\bigtriangleup u_i=\left(\frac{1}{4}+(t_i)^2 \right)u_i \ \ \text{and} \ \ T_nu_i=\lambda_i(n)u_i,$$
where $0<t_1\leq t_2 \leq\dots,$ and $\lambda_i(n) \in \R$. 
From the Fourier expansion of Maass form, for $z=x+iy \in \mathcal{H},$
$$u_i(z)=\sqrt{y} \rho_i(1) \sum_{n  \neq 0} \lambda_i(n)K_{it_i}(2 \pi |n|y)e(nx),$$
where $\rho_i(1) \neq 0$ and $K_v$ is the $K$-Bessel function of order $v.$
Moreover we know :
$$\Omega(T):= \#\{i:0<t_i \leq T\}=\frac{1}{4 \pi} vol(\Gamma / \mathcal{H})T^2 + \O(T \log  T).$$
The Ramanujan conjecture predicts that for any prime $p,$
$$|\lambda_i(p)| \leq 2.$$
At present we are far from above bound. The best bound towards Ramanujan's conjecture for Maass forms is due to Kim and Sarnak (see~\cite{KS}), that is for any prime $p$,
$$|\lambda_i(p)| \leq p^{\theta}+p^{- \theta},$$
where $\theta=\frac{7}{64}.$  Note that, the conjecture predicts $\theta=0.$  
However, in~\cite[Theorem 1.1]{Sarnak},  Sarnak showed that for a fixed prime $p$, 
 $$\# \{i \leq \Omega(T): |\lambda_i(p)|\geq \alpha \geq 2\} \ll T^{2-\frac{\log \frac{\alpha}{2}}{\log \,p}}$$
 and analogously in~\cite[Lemma 4.3]{LY} Lau and Wang showed that the above number is
 $$\ll\Omega(T)\left(\frac{\log \,p}{\log T}\right)^2.$$
 In particular, almost all eigenvalues $\lambda_i(p)$ (in the sense of density in $j$) lie in the interval $[-2,2].$ 
    Let $f$ be a primitive Maass form.  Assuming Ramanujan's conjecture, write
$$\lambda_i(p)=2 \cos \theta_{i}(p), \ \ \theta_{i}(p) \in [0,\pi].$$
In this case the $\theta_{i}(p)$ are called the eigenangles.
The following result can be seen as a special case of~\cite[Theorem 1.2 ]{Sarnak} and~\cite[Theorem 1]{LY}: \\
Let $\mu_p$ be the measure as defined in (\ref{E50}).
Then the family 
$$\{\lambda_i(p), 1 \leq i \leq \Omega(T)\}$$ is equidistributed with respect to the measure $\mu_p.$ \\
Let $n=\Omega(T)$ and for all $1 \leq j \leq r,$
$$A_{n}^{(j)}=\left \{\frac{\theta_{i}^{(j)}(p)}{2 \pi}, \ 1\leq i \leq \Omega(T) \right \}.$$
Note that $$\# A_{n}^{(j)}=n=\Omega(T).$$
So as $N+k \to \infty,$ we have $n \to \infty$ that is we have infinite number of multi sets.  Since each sets are same and we are going to study the distributions of gaps between the elements of the multisets $A_{n}^{(j)}.$
$$A_{n}^{(j)}=\left \{\frac{\theta_{i}^{(j)}(p)}{2 \pi}, \ 1\leq i \leq \Omega(T) \right \},$$
where $\theta_{i}^{(j)}(p) \in [0,\pi]$ such that $\lambda_{i}^{(j)}(p)= 2 \cos \theta_{i}(p).$
\begin{thm}\label{T7}
 There exist a small constant  $\delta >0$ such that for all large $T,$
$$\frac{1}{(\Omega(T))^r}\# \left\{0< t_{i}^{j} \leq T, 1 \leq j \leq r: \left\{\frac{\pm \theta_{i}^{(1)}(p)\pm  \theta_{i}^{(2)}(p)\dots\pm \theta_{i}^{(r)}(p)}{2 \pi}\right\} \in I \right\}$$
$$ =\int_I \nu_p dx+ \O \left(\frac{ \log \, p}{\log \, T}\right),$$
holds uniformly for integers $  1 \leq r \leq \Omega(T)$ and for prime $p$ satisfying 
 $$r \log \, p \leq  \delta \log  \, T,$$
 and uniformly
 for any interval $I=[a,b] \subseteq [0,1]$.
 Here 
 $$\nu_p= \underbrace{F_p(x)*F_p(x)*\dots*F_p}_\text{$r$ times}(x),$$
 and
 $F_p(x)$ is as defined in (\ref{E50}).
 Here the implied constant is effectively computable.
 \end{thm}
 As a consequence of above theorem, we have the following theorem
\begin{thm}\label{T24}
For any $\alpha \in [0, 1]$,
$$\# \left\{0< t_{i}^{j} \leq T, 1 \leq j \leq r: \left(\frac{\pm \theta_{i}^{(1)}(p)\pm  \theta_{i}^{(2)}(p)\dots\pm \theta_{i}^{(r)}(p)}{2 \pi}= \alpha\right)  \right\}$$
$$= \O \left((\Omega(T))^r \left(\frac{ \log  p}{\log  T}\right)\right).$$
\end{thm}
\begin{remark}
For $r=2$, the above theorem gives similar result like Theorem \ref{T15}.
\end{remark}
\begin{proof}
As a special case of~\cite[Theorem 1.2 ]{Sarnak} and~\cite[Theorem 1]{LY} we can conclude that the family
$\left \{\frac{\pm \theta_{i}^{(j)}(p)}{2 \pi}, 1 \leq i \leq \Omega(T)\right \}$ is equidistributed in $[- \frac{1}{2}, \frac{1}{2}]$ with respect to $F_p(x)dx.$ 
 To proceed further, we prove a Proposition like Propositions \ref{P10} and \ref{P11}.
 \begin{prop}\label{P12}
 For $m \geq 1, \ 1 \leq  i \leq \Omega(T)$, $1 \leq j \leq r$, $0<k<\frac{11}{155}$ and $\eta>\frac{43}{620}$,
$$\left|\prod_{j=1}^{r}\sum_{i=1}^{\Omega(T)} \cos m \theta_{i}^{(j)}(p)- \Omega(T)^{r}C_m \right|$$
$$\ll(T^{2-k}p^{m{\eta}})^r,$$
where the implied constant depends only on $\eta.$  
 \end{prop}
\begin{proof}
The following result is a special case of~\cite[Lemma 4.1]{LY}: \\
 For any positive integer $m$, let $c_{m}$ be the Weyl limits of the family $\left \{\frac{\pm \theta_{i}^{(j)}(p)}{2 \pi}, 1 \leq i\leq \Omega(T)\right \}.$
Then the Weyl limits $c_{m}$ are given by 
 \begin{equation}\label{E116}
c_{m} =  
\left \{ 
\begin{array}{l l}
  1 &  \text{if } m=0 \\
  \left(\frac{1}{p^{\frac{m}{2}}}-\frac{1}{p^{{\frac{m-2}{2}}}}\right) &  \text{if $m$ is even }  \\
  0 &  \text{otherwise.}
\end{array}      
\right.
\end{equation}
    Moreover, for $m \geq 1,$
    $$\left|\sum_{i=1}^{\Omega(T)}  2 \cos m \theta_{i}^{(j)}(p) -c_m \Omega(T)\right|\ll T^{2-k}p^{m \eta}.$$
    Lau and Wang~\cite{LY} proved the above result, using Kuznetsov trace formula.
    Now proceeding like proof of Propositions \ref{P10} and \ref{P11}, we get Proposition \ref{P12}.
\end{proof}
Now proceeding just like proof of Theorems \ref{pair-correlation} and \ref{T13} and choosing 
$M=c \frac{\log \,p}{\log  \, T}$, we have the required result.
\end{proof}
\noindent\textbf{Acknowledgments:}
The author would like to thank Dr. Kaneenika Sinha for many useful discussion, comments and correction in earlier versions of the paper.  The author also thanks Prof. Ram Murty for helpful comments and providing an earlier version of~\cite{MSr}.  We would also like to thank Dr. Baskar Balasubramanyam for corrections in an earlier version of the paper.

\end{document}